\documentclass[10pt]{article}
\usepackage[english]{babel}
\usepackage{amsfonts}
\usepackage{amsthm}
\usepackage{amsmath}
\usepackage{amssymb}
\input xy
\xyoption{all}

\DeclareMathOperator{\Aut}{Aut}

\DeclareMathOperator{\Sym}{Sym}
\DeclareMathOperator{\Alt}{Alt}
\DeclareMathOperator{\soc}{soc}

\title{Finite Groups that are the union of at most $25$ proper subgroups}
\author{Martino Garonzi, \\Dipartimento di Matematica Pura ed Applicata, \\Via Trieste 63, 35121 Padova, Italy. \\E-mail address: mgaronzi@studenti.math.unipd.it}
\date{}
\begin{document}
\maketitle
\newtheorem{defi}{Definition}
\newtheorem{teor}{Theorem}
\newtheorem{cor}{Corollary}
\newtheorem{prop}{Proposition}
\newtheorem{lemma}{Lemma}
\newtheorem{oss}{Observation}
\begin{abstract}
For a finite group $G$ let $\sigma(G)$ (the ``sum'' of $G$) be the least number of proper subgroups of $G$ whose set-theoretical union is equal to $G$, and $\sigma(G)=\infty$ if $G$ is cyclic. We say that a group $G$ is $\sigma$-elementary if for every non-trivial normal subgroup $N$ of $G$ we have $\sigma(G)<\sigma(G/N)$. In this article we produce the list of all the $\sigma$-elementary groups of sum up to $25$. We also show that $\sigma(\Aut(PSL(2,8)))=29$.
\end{abstract}
\section{Introduction} \label{i}
Let $G$ be a finite group. If $G$ is not cyclic we can consider the ``covers'' of $G$, the families of proper subgroups of $G$ whose set-theoretical union is equal to $G$, and we can define $\sigma(G)$ (the ``sum'' of $G$, a concept introduced by Cohn in 1994: see \cite{cohn}) to be the least cardinality of a cover of $G$, i.e. the cardinality of a minimal cover of $G$. If $G$ is cyclic then the sum is not well defined because no proper subgroup contains any generator of $G$; in this case we define $\sigma(G)=\infty$, with the convention that $n < \infty$ for every integer $n$. An easy result is that if $N \unlhd G$ then $\sigma(G) \leq \sigma(G/N)$, because every cover of $G/N$ corresponds to a cover of $G$.
\begin{defi}[$\sigma$-elementary groups]
$G$ is said to be $\sigma$-elementary if for every non trivial normal subgroup $N$ of $G$ we have $\sigma(G)<\sigma(G/N)$. $G$ is said to be $n$-elementary if $G$ is $\sigma$-elementary and $\sigma(G)=n$.
\end{defi}
It is an easy exercise to show that $\sigma(G) \neq 2$ for every group $G$. In this article we find all the $n$-elementary groups for every $3 \leq n \leq 25$. We will produce an explicit tabular of them. In particular we obtain that:
\begin{teor}
A finite $\sigma$-elementary non-abelian group $G$ with $\sigma(G) \leq 25$ is either of affine type or almost-simple with socle of prime index.
\end{teor}
In 1926 Scorza proved that $\sigma(G)=3$ if and only if $C_2 \times C_2$ is an epimorphic image of $G$ (cf. \cite{scorza}): in our terminology $C_2 \times C_2$ is the unique $3$-elementary group. Cohn in 1994 found all the $\sigma$-elementary groups with sum up to $6$ (cf. \cite{cohn}). Tomkinson in 1997 found the sum of the solvable groups: he found that if $G$ is a finite solvable non-cyclic group then $\sigma(G)=q+1$ where $q$ is the least order of a chief factor of $G$ with more than a complement (cf. \cite{tomkinson}). It turned out that it is interesting to ask whether a sum can be not of the form $q+1$ with $q$ a prime power. Tomkinson found the answer for the first such integer, $7$: he showed that $\sigma(G) \neq 7$ for every finite group $G$, i.e. that $7$ is not a sum (cf. \cite{tomkinson}). Some time later the next three integer of this kind, $11$, $13$ and $15$, were solved: in 1999, R.A. Bryce, V. Fedri and L. Serena found that $\sigma(PSL(3,2))=15$ (cf. \cite{BFS}); in 2007 Abdollahi, Ashraf, Shaker showed that $\sigma(\Sym(6))=13$ (cf. \cite{shaker}); in 2008 Lucchini and Detomi showed that $11$ is not a sum (cf. \cite{lucchini}). The integers between $16$ and $25$ which are not of the form $q+1$ with $q$ a prime power are $16,19,21,22,23,25$. In this work we find in particular that:
\begin{teor}
The integers between $16$ and $25$ which are not sums are $19,21,22,25$.
\end{teor}
We also obtain the following:
\begin{teor}
$\sigma(\Aut(PSL(2,8)))=29$.
\end{teor}
In sections \ref{intro} and \ref{solv} we will recall some known results on the structure of the $\sigma$-elementary groups. With the help of these results we will be able to prove that if $n \leq 25$ then an $n$-elementary group has a primitive permutation representation of degree smaller than $n$. Therefore a crucial part of our proof is the study of $\sigma(G)$ for $G$ a primitive permutation group of low degree: all the needed results in this direction will be collected in section \ref{prim}.
\section{Preliminary results} \label{intro}
Let $G$ be a finite group. We recall the definition of the primitive monolithic group $X$ associated to a non-Frattini minimal normal subgroup $N$ of $G$. We consider two cases:
\begin{itemize}
\item $N$ is abelian. Since $N$ is non-Frattini there exists a complement $H$ of $N$ in $G$. Then we define $X:=N \rtimes H/C_H(N)$.
\item $N$ is non-abelian. In this case we define $X:=G/C_G(N)$.
\end{itemize}
In any case $X$ is a primitive monolithic group with socle isomorphic with $N$. \\
Except for the result about the center (which is proved in \cite{cohn}), the following results can be found in \cite{lucchini}:
\begin{teor} \label{sigmaprim}
Let $G$ be a finite, non-abelian and $\sigma$-elementary group.
\begin{enumerate}
\item $\Phi(G)=Z(G)=1$;
\item $G$ has at most one abelian minimal normal subgroup;
\item let $\soc(G)=G_1 \times ... \times G_n$ be the socle of $G$, where $G_1,...,G_n$ are the minimal normal subgroups of $G$. Then $G$ is a subdirect product of the primitive monolithic groups $X_i$ associated to the $G_i$'s. In particular every $X_i$ is an epimorphic image of $G$.
\end{enumerate}
\end{teor}
\begin{prop} \label{1e2}
Let $G$ be a non-abelian $\sigma$-elementary group, $N$ a minimal normal subgroup of $G$ and $X$ the primitive monolithic group associated to $N$. Then:
\begin{itemize}
\item If $X=N$ then $G=X=N$.
\item If $|X/N|$ is a prime then $G=X$.
\end{itemize}
\end{prop}
\begin{prop} \label{cohom}
Let $H$ be a group, and let $V$ be a $H$-module. Define $G:=V \rtimes H$ and suppose that $C_V(H)=0$ (this is the case if $V$ is a non-central minimal normal subgroup of $G$). Then:
\begin{enumerate}
\item If $H^1(H,V) \neq 0$ then $\sigma(G)=\sigma(H)$;
\item if $\sigma(H) \geq 2|V|$ then $H^1(H,V)=0$.
\end{enumerate}
\end{prop}
\section{About solvable $\sigma$-elementary groups} \label{solv}
In this paragraph we will give and prove a known result about solvable $\sigma$-elementary groups, which is a consequence of the next result of Tomkinson (cf. \cite{tomkinson}):
\begin{teor}[Tomkinson] \label{tom}
Let $G$ be a solvable non-cyclic group. Then $\sigma(G)=|S/K|+1$ where $|S/K|$ is the least order of a chief factor of $G$ with more than a complement.
\end{teor}
Let $G$ be a non-abelian solvable $\sigma$-elementary group. By theorem \ref{sigmaprim} we know that $G$ is monolithic (let $V$ be its socle), and by theorem \ref{tom} we have $\sigma(G)=|S/K|+1$ for a chief factor $S/K$ of $G$ with multiple complements, and whose order is minimal among the orders of the chief factors with this property. Suppose $K \neq 1$, i.e. $V \leq K$. If $G/V$ is not cyclic then $$|S/K|+1=\sigma(G)<\sigma(G/V)=1+|S_0/K_0|,$$ where $S_0/K_0$ is a smallest chief factor of $G/V$ with more than a complement. The inequality $|S/K|<|S_0/K_0|$ gives then a contradiction, and $K=1$. This means that $\sigma(G)=|V|+1$. Since $G/V$ acts faithfully and irreducibly on $V$, by the result of Gasch\"utz (see \cite{gaschutz}) the chief factors of $G/V$ have order $<|V|$, and this implies $|V|+1=\sigma(G) \leq \sigma(G/V)<|V|+1$, contradiction. We deduce that $G/V$ is cyclic, and $\sigma(G)=|V|+1$ by lemma 2.1 in \cite{tomkinson}. \\
Summarizing:
\begin{teor} \label{risol}
Let $G$ be a finite solvable non-abelian group. Then $G$ is $\sigma$-elementary if and only if it is monolithic and $G/\soc(G)$ is cyclic. In this case $\sigma(G)=|\soc(G)|+1$.
\end{teor}
\section{About the sum of some primitive groups} \label{prim}
We study in this section the sum of primitive groups of degree up to $24$. This will simplify the study of $n$-elementary groups. We begin with a known lemma.
\begin{lemma} \label{app}
If $H$ is a maximal subgroup of a group $G$ and $\sigma(H)>\sigma(G)$ then $H$ appears in every minimal cover of $G$. In particular if $H$ is maximal and non-normal then $\sigma(H) < [G:H]$ implies $\sigma(G) \geq \sigma(H)$.
\end{lemma}
\begin{proof}
Let $\{H_1,...,H_n\}$ be a minimal cover of $G$, where $n=\sigma(G)$. Then $H=(H \cap H_1) \cup ... \cup (H \cap H_n)$ is an union of less than $\sigma(H)$ subgroups equal to $H$, so at least one of them must be unproper (by definition of $\sigma(H)$): $H \cap H_i=H$ for some $i \in \{1,...,n\}$, i.e. $H=H_i$ since $H$ is maximal. Every conjugate of $H$ is a maximal subgroup of $G$ isomorphic to $H$, thus if $H$ is not normal then in every minimal cover of $G$ there are all the $[G:H]$ conjugates of $H$.
\end{proof}
\subsection{Primitive groups with non-abelian socle}
In this paragraph we will prove the bounds on the sums of the primitive groups of degree up to $24$ and non-abelian socle which are stated in the following tabular.
$$\begin{tabular}{|c|c|}
\hline $\deg$ & $\sigma$ \\ \hline $5$ & $\sigma(\Alt(5))=10;\sigma(\Sym(5))=16.$ \\ \hline $6$ & $\sigma(\Alt(5))=10;\sigma(\Sym(6))=13;\sigma(\Sym(5))=16.$ \\ \hline $7$ & $\sigma(SL(3,2))=15;\sigma(\Alt(7))=31;\sigma(\Sym(7))=64.$ \\ \hline $8$ & $\sigma(PSL(2,7))=\sigma(PGL(2,7))=29;$ \\ & $\sigma(\Alt(8)) \geq 64;\sigma(\Sym(8)) \geq 29.$ \\ \hline $9$ & $\sigma(\Aut(PSL(2,8)))=29;\sigma(PSL(2,8))=36;$ \\ & $\sigma(\Alt(9)) \geq 80;\sigma(\Sym(9)) \geq 172.$ \\ \hline $10$ & $\sigma(\Alt(5))=10;\sigma(\Sym(6))=13;\sigma(\Alt(6))=16;$ \\ & $\sigma(PGL(2,9))=46;\sigma(M_{10}),\sigma(\Alt(10)),\sigma(\Sym(10)) \geq 45,\sigma(P \Gamma L(2,9))=3.$ \\ \hline $11$ & $\sigma(M_{11})=23;\sigma(PSL(2,11))=67;\sigma(\Alt(11)),\sigma(\Sym(11)) \geq 512.$ \\ \hline $12$ & $\sigma(M_{11})=23;\sigma(PSL(2,11))=\sigma(PGL(2,11))=67;$ \\ & $\sigma(M_{12}),\sigma(\Alt(12)),\sigma(\Sym(12)) \geq 67.$ \\ \hline $13$ & $\sigma(PSL(3,3)),\sigma(\Alt(13)),\sigma(\Sym(13)) \geq 144.$ \\ \hline $14$ & $\sigma(PSL(2,13)),\sigma(PGL(2,13)),\sigma(\Alt(14)),\sigma(\Sym(14)) \geq 92.$ \\ \hline $15$ & $\sigma(\Sym(6))=13;\sigma(\Alt(6))=16;\sigma(\Alt(7))=31;$ \\ & $\sigma(PSL(4,2)),\sigma(\Alt(15)),\sigma(\Sym(15)) \geq 64.$ \\ \hline $16$ & $\sigma(\Alt(16)) \geq 2^{14}; \sigma(\Sym(16))>6435.$ \\ \hline $17$ & $\sigma(PSL(2,16)),\sigma(PSL(2,16):2,$ \\ & $\sigma(P \Gamma L(2,16)),\sigma(\Alt(17)),\sigma(\Sym(17)) \geq 68.$ \\ \hline $18$ & $\sigma(PSL(2,17)),\sigma(PGL(2,17)),\sigma(\Alt(18)),\sigma(\Sym(18)) \geq 2^{16}.$ \\ \hline $19$ & $\sigma(\Alt(19)),\sigma(\Sym(19)) \geq 2^{17}$ \\ \hline $20$ & $\sigma(PSL(2,19)),\sigma(PGL(2,19)),\sigma(\Alt(20)),\sigma(\Sym(20)) \geq 191.$ \\ \hline $21$ & $\sigma(SL(3,2))=15;\sigma(\Alt(7))=31;\sigma(P \Gamma L(3,4))=3;$ \\ & $\sigma(PGL(2,7)),\sigma(\Sym(7)),\sigma(PSL(3,4)),\sigma(P \Sigma L(3,4)),$ \\ & $\sigma(PGL(3,4)),\sigma(\Alt(21)),\sigma(\Sym(21)) \geq 64.$ \\ \hline
\end{tabular}$$
$$\begin{tabular}{|c|c|}
\hline $\deg$ & $\sigma$ \\ \hline $22$ & $\sigma(M_{22}),\sigma(M_{22}:2),\sigma(\Alt(22)),\sigma(\Sym(22)) \geq 67.$ \\ \hline $23$ & $\sigma(M_{23}),\sigma(\Alt(23)),\sigma(\Sym(23)) \geq 64.$ \\ \hline $24$ & $\sigma(M_{24}),\sigma(PSL(2,23)),\sigma(PGL(2,23)),$ \\ & $\sigma(\Alt(24)),\sigma(\Sym(24)) \geq 277.$ \\ \hline
\end{tabular}$$
The bounds about $\sigma(\Sym(n))$ when $n$ is odd or $n \geq 14$, $\sigma(\Alt(n))$ when $n \neq 7$ and $\sigma(M_{11})$ are proved in the work of Maroti \cite{maroti}. $\sigma(\Sym(5))$ and $\sigma(\Alt(5))$ are computed by Cohn in \cite{cohn}, $\sigma(\Sym(6))$ is found by Abdollahi, Ashraf and Shaker in \cite{shaker}, the bound on $\sigma(M_{10})$ is proved by Lucchini and Detomi in \cite{lucchini}, the sums of the groups of the form $PGL(2,q)$ and $PSL(2,q)$ are found by Bryce, Fedri and Serena in \cite{BFS}, $\sigma(\Alt(7))$ is found by Kappe in \cite{kappe}. The only groups that we are left to study to prove what is stated in the tabular are the following:
\begin{itemize}
\item $\Sym(8)$. It admits $PGL(2,7)$ as a maximal subgroup of index $120$ and sum $29$, so $\sigma(\Sym(8)) \geq 29$ by lemma \ref{app}.
\item Let $G:=\Aut(PSL(2,8)) = P \Gamma L(2,8)$. We are going to show that $\sigma(G)=29$ and that there exists only one minimal cover of $G$. \\
$G$ is an almost simple group of order $1512=2^3 \cdot 3^3 \cdot 7$. $PSL(2,8)$, its non-trivial proper normal subgroup, is a maximal subgroup of sum $36$, thus if $\sigma(G)<36$ then $PSL(2,8)$ appears in every minimal cover. \\
Now, since $\soc(G)$ together with the normalizers of the $3$-Sylow subgroups of $PSL(2,8)$ form a cover of $G$ consisting of $29$ subgroups, we have $\sigma(G) \leq 29$ so $\soc(G)$ appears in every minimal cover of $G$. The only maximal subgroups of $G$ which contain elements of order $9$ are $\soc(G)=PSL(2,8)$ and the normalizers of the $3$-Sylow subgroups of $PSL(2,8)$. It follows that if $P$ is a $3$-Sylow subgroup of $\soc(G)=PSL(2,8)$ then $N_G(P)$ is the only maximal subgroup of $G$ which contains the elements of order $9$ in $N_G(P)-P$. So the $28$ normalizers of the $3$-Sylow subgroups of $\soc(G)$ appear in every minimal cover. So $\sigma(G)=29$.
\item $\Sym(10)$. Its maximal subgroups (cf. \cite{atlas}) are:
\begin{itemize}
\item $\Sym(9)$;
\item $\Sym(8) \times C_2$;
\item $\Sym(7) \times \Sym(3)$;
\item $(\Sym(5) \times \Sym(5)):2$;
\item $\Sym(6) \times \Sym(4)$;
\item $2^5:\Sym(5)$;
\item $\Alt(6):2^2$.
\end{itemize}
$\Sym(8)$ and $\Sym(9)$ do not have elements of order $21$. Thus $\Sym(8) \times C_2$ has no elements of order $21$, and the only maximal subgroups of $\Sym(10)$ which contain elements of order $21$ are of the kind $\Sym(7) \times \Sym(3)$. Now $\Sym(10)$ has $6! \cdot \binom{10}{7} \cdot 2! = 172800$ elements of order $21$, and $\Sym(7) \times \Sym(3)$ has $6! \cdot 2! = 1440$ such elements, so in order to cover the elements of order $21$ we need at least $172800/1440 = 120$ proper subgroups. Therefore $\sigma(\Sym(10)) \geq 120$.
\item $M_{12}$. It admits $PSL(2,11)$ as a maximal non normal subgroup of sum $67$ and index $144$, so $\sigma(M_{12}) \geq 67$ by lemma \ref{app}.
\item $\Sym(12)$. It admits $PGL(2,11)$ as a maximal non normal subgroup of sum $67$ and index $9!$, so $\sigma(\Sym(12)) \geq \sigma(PGL(2,11))=67$ by lemma \ref{app}.
\item $PSL(3,3)$. It has $1728$ elements of order $13$, and its maximal subgroups which contain elements of order $13$ are of the kind $C_{13} \rtimes C_3$, and such subgroups are $144$. Since the $C_{13} \rtimes C_3$ have $12$ elements of order $13$, in order to cover the elements of order $13$ we need at least $1728/12=144$ proper subgroups. In particular $\sigma(G) \geq 144$.
\item $G:=PSL(2,16):2$ admits $PSL(2,16)$ as a maximal and normal subgroup. The maximal subgroups of $G$ are:
\begin{itemize}
\item $17$ occurs of $((2^4.5).3).2$;
\item $120$ occurs of $17.4$;
\item $68$ occurs of $C_2 \times \Alt(5)$;
\item $1$ occur of $PSL(2,16)$;
\item $136$ occurs of $\Sym(3) \times D_{10}$.
\end{itemize}
The only maximal subgroups which contain elements of order $10$ are $C_2 \times \Alt(5)$ (which has $24$ elements of order $10$) and $\Sym(3) \times D_{10}$ (which contains $12$ elements of order $10$), and since $G$ contains $1632$ elements of order $10$, we need at least $1632/24=68$ proper subgroups to cover the elements of order $10$. In particular $\sigma(G) \geq 68$.
\item $G:=P \Gamma L(2,16) = PSL(2,16) \rtimes C_4$. The maximal subgroups of $G$ are:
\begin{itemize}
\item $17$ occurs of $((2^4.5).3).4$;
\item $136$ occurs of $(5.4) \times \Sym(3)$;
\item $68$ occurs of $\Alt(5).4$;
\item $120$ occurs of $17.8$;
\item $1$ occur of $PSL(2,16) \rtimes C_2$.
\end{itemize}
The only maximal subgroups which contain elements of order $12$ are $(5.4) \times \Sym(3)$ (which contains $20$ elements of order $12$) and $\Alt(5).4$ (which contains $40$ elements of order $12$), so to cover the elements of order $12$ (which are $2720$) we need at least $2720/40=68$ proper subgroups, so that $\sigma(G) \geq 68$.
\item $PSL(3,4)=M_{21}$. The only maximal subgroup of $PSL(3,4)$ which contains elements of order $7$ is $PSL(2,7)$, which contains $48$ elements of order $7$. Since $PSL(3,4)$ has $5760$ elements of order $7$, the sum of $PSL(3,4)$ is at least $5760/48=120$.
\item $P \Sigma L(3,4)=PSL(3,4):2$. The only maximal subgroup of $P \Sigma L(3,4)$ which contains elements of order $14$ is $PSL(2,7) \times C_2$, which contains $48$ elements of order $14$. Since $P \Sigma L(3,4)$ has $5760$ elements of order $14$, the sum of $P \Sigma L(3,4)$ is at least $120$.
\item $PGL(3,4)$. The only maximal subgroup of $PGL(3,4)=PSL(3,4):3$ which contains elements of order $21$ is $(7:3) \times 3$, which contains $12$ elements of order $21$. Since $PGL(3,4)$ contains $11520$ elements of order $21$, the sum of $PGL(3,4)$ is at least $11520/12=960$.
\item $P \Gamma L(3,4)$. It is a $3$-sum group.
\item $M_{22}$. It contains $80640$ elements of order $11$, and the only maximal subgroup of $M_{22}$ which contains elements of order $11$ is $PSL(2,11)$, which contains $120$ such elements. Thus the sum of $M_{22}$ is at least $80640/120 = 672$.
\item $M_{22}:2$. It admits $PGL(2,11)$ as a maximal and non normal subgroup of index $672$ and sum $67$, so $\sigma(M_{22}:2) \geq 67$ by lemma \ref{app}.
\item $M_{23}$. It admits $\Alt(8)$ as a maximal non normal subgroup of index $506$ and sum $\geq 64$, so $\sigma(M_{23}) \geq 64$ by lemma \ref{app}.
\item $M_{24}$. It admits $PSL(2,23)$ as a maximal non normal subgroup of sum $277$ and index $40320$, so $\sigma(M_{24}) \geq 277$ by lemma \ref{app}.
\end{itemize}
\subsection{Primitive groups with abelian socle}
The sums or some bounds on the sums of the primitive groups of degree up to $24$ and abelian socle are summarized in the following tabular. The solvable primitive groups which are not $\sigma$-elementary and the cyclic groups are omitted. If $p$ is a prime, the subgroups of $AGL(1,p)$ which contain the socle $\mathbb{F}_p$ are $p$-primitive and $\sigma$-elementary of sum $p+1$ (this follows from section \ref{solv}); they are omitted.
$$\begin{tabular}{|c|c|}
\hline $\deg$ & $\sigma$ \\ \hline $4$ & $\sigma(\Alt(4))=5.$ \\ \hline $8$ & $\sigma(AGL(1,8))=9; \sigma(ASL(3,2))=15.$ \\ \hline $9$ & $\sigma(3^2:4)=\sigma(AGL(1,9))=10.$ \\
\hline $16$ & $\sigma(AGL(4,2)) \geq 31, \sigma(A \Gamma L(2,4)) \leq 4, \sigma(ASL(2,4):2) \leq 16,$ \\ & $\sigma(AGL(2,4)) \leq 10, \sigma(ASL(2,4)) \leq 10, \sigma(2^4:\Sym(6)) \leq 13,$ \\ & $\sigma(2^4:\Alt(6)) \leq 16, \sigma(2^4:\Sym(5)) \leq 16, \sigma(2^4:\Alt(5)) \leq 10,$ \\ & $\sigma(2^4:\Alt(7)) = 31, \sigma(2^4:5)=\sigma(AGL(1,16))=17.$ \\ \hline
\end{tabular}$$
To show this recall theorem \ref{risol}: if a non-cyclic solvable and primitive monolithic group $G$ is such that $G/\soc(G)$ is cyclic then $G$ is $\sigma$-elementary and $\sigma(G)=|\soc(G)|+1$. This solves all the cases except for $ASL(3,2)$ and the $16$-primitive groups. We have $\sigma(ASL(3,2))=15$ by proposition \ref{cohom} because $H^1(SL(3,2),{\mathbb{F}_2}^3)=C_2$ (in fact in \cite{coh} it is shown that $H^1(GL(d,2),{\mathbb{F}_2}^d)=0$ if $d \neq 3$ and $H^1(SL(3,2),{\mathbb{F}_2}^3)=C_2$) and $\sigma(SL(3,2))=15$ (by \cite{BFS}). The $16$-primitive groups with abelian socle are of the kind $V \rtimes H$ with $H \leq GL(V)$ and $V={C_2}^4$. In almost all the cases the bounds given in the tabular are of the kind ``$\sigma(V \rtimes H) \leq \sigma(H)$'' using for $\sigma(H)$ the information collected in the previous section. Only two cases require more attention:
\begin{itemize}
\item $G:=AGL(4,2)={\mathbb{F}_2}^4 \rtimes GL(4,2)$. We want to show that $\sigma(G) \geq 31$. We observe that $GL(4,2) \cong \Alt(8)$ and $G$ is monolithic. Let $V:={\mathbb{F}_2}^4$ and $H:=GL(4,2)$, so that $G=V \rtimes H$. Since $\sigma(\Alt(8)) \geq 69$, if (as we can suppose) $\sigma(G)<69$ then every complement of $V$ must appear in every minimal cover $\{M_1,...,M_n\}$ of $G$, where $n=\sigma(G)$. We have $H^1(H,V)=0$ by the result in \cite{coh}, so that $V$ has exactly $16$ complements in $G$, let them be $M_1,...,M_{16}$. If $g \in GL(4,2)$ stabilizes a non-zero vector $v \in V$ then the function ${\mathbb{F}_2}^4 \to {\mathbb{F}_2}^4$ which sends $x$ to $x^g-x$ is not injective (having $v$ in its kernel), so it is not surjective: there exists $w \in V$ such that $x^g-x \neq w$ for every $x \in V$. In this case $wg \in G$ does not belong to any complement of $V$, because the complements of $V$ are conjugate to $H$, and $wg \in H^x$ means $g \in H$ and $w=x-x^g$. Thus $wg$ must lie in at least one $M_i$ with $i \geq 17$, so $g$ must belong to it, because the maximal subgroups which do not complement $V$ must contain it ($V$ is the only minimal normal subgroup of $G$). It turns out that $M_{17}/V,...,M_n/V$ must cover all the point stabilizers of $GL(4,2)$. If $v$ is a non-zero vector then the point stabilizer of $v$ in $GL(4,2)$ is isomorphic to $ASL(3,2)$, which is a $15$-sum group. So either every point stabilizer is one of the $M_i/V$ with $i \geq 17$ or the $M_i/V$ with $i \geq 17$ are at least $15$. In any case the $M_i/V$ with $i \geq 17$ are at least $15$, so $\sigma(G) \geq 15+16=31$.
\item $G:={\mathbb{F}_2}^4 \rtimes \Alt(7)$. Let $V:={\mathbb{F}_2}^4$ and $H:=\Alt(7)$. We want to show that $\sigma(G) \geq 31$. Suppose by contradiction that $\sigma(G) \leq 30$, so that all the complements of $V$ appear in every minimal cover of $G$ (because $\sigma(\Alt(7))=31$). Since $\sigma(G) \leq \sigma(H)=31$ we have $H^1(H,V)=0$ (otherwise we would have at least $32$ complements of $V$). Let $M_1,...,M_{16}$ be the $16$ complements of $V$ in $G$. Since $H$ acts faithfully on $16$ elements, we have an injection $\Alt(7) \to \Sym(16)$, and a $7$-cicle $h$ of $\Alt(7)$ must fix a non-zero vector in this action (the image of a $7$-cycle in $\Sym(16)$ is either a $7$-cycle or a product of two disjoint $7$-cycles), let $v$ be this vector. Then $vh \in G$ does not belong to any complement of $H$ because it has order $14$, and $H$ has no elements of order $14$. Thus $vh$ lies in a $M_i$ with $i \geq 17$, and every such $M_i$ contains $V$ (because it does not complement it), so $M_{17}/V,...,M_n/V$ contain together all the $7$-cycles. But to cover the $7$-cycles in $\Alt(7)$ we need at least $15$ subgroups, because the $7$-cycles in $\Alt(7)$ are $6!$ and the only maximal subgroups of $\Alt(7)$ which contain $7$-cycles are the $SL(3,2)$, and they contain $48$ $7$-cycles. We deduce that $\sigma(G) \geq 16+15=31$, contradiction. In particular since $\sigma(G) \leq \sigma(\Alt(7))=31$, $G$ is a non-$\sigma$-elementary $31$-sum group.
\end{itemize}
\section{The $n$-elementary groups with $n \leq 25$} \label{nel}
If $X$ is a primitive monolithic group and $N:=\soc(X)$, we denote by $l_X(N)$ the minimal index of a proper supplement of $N$ in $X$. We recall that $X$ is $l_X(N)$-primitive. In particular if $N$ is abelian then the supplements of $N$ are in fact complements, so in this case $l_X(N)=|N|$ and $X$ is $|N|$-primitive. \\
The following lemma summarizes some known results which can be found in \cite{lucchini}.
\begin{lemma} \label{tanto}
If $G$ is an abelian non-cyclic $\sigma$-elementary group then $G \cong C_p \times C_p$ and $\sigma(G)=p+1$, for some prime $p$. If $G$ is non-abelian, $G_1,...,G_n$ are its minimal normal subgroups and $X_1,...,X_n$ are the associated primitive monolithic groups then:
\begin{enumerate}
\item $\sigma(G) \geq \sum_{i=1}^n l_{X_i}(G_i)$.
\item If $G_1$ is abelian then $\frac{1}{2} \sigma(G) < |G_1| \leq \sigma(G)-1$.
\end{enumerate}
\end{lemma}
With the following lemma we reduce our study to monolithic groups.
\begin{lemma} \label{33}
The non-abelian $\sigma$-elementary groups of sum $\leq 33$ are primitive and monolithic.
\end{lemma}
\begin{proof}
Let $G$ be a non-abelian $\sigma$-elementary group. Let $G_1,...,G_n$ be the minimal normal subgroups of $G$, and let $X_1,...,X_n$ be the primitive monolithic groups associated to $G_1,...,G_n$ respectively. We have then $\sigma(G) \geq \sum_{i=1}^n l_{X_i}(G_i)$, and $l_{X_i}(G_i)=|G_i|$ if $G_i$ is abelian. We may suppose by theorem \ref{sigmaprim} (2) that $G_2,...,G_n$ are not abelian. \\
Consider a non-abelian minimal normal subgroup $G_i$. In the following discussion we will apply proposition \ref{1e2}.
\begin{itemize}
\item If $t:=l_{X_i}(G_i) \leq 16$ and $t \neq 6,10$ then since $X_i$ is $t$-primitive we immediately see that $X_i/G_i$ has order $\leq 3$, so either $X_i=G_i$ or $X_i/G_i$ is a prime, and this implies $G=X_i$ and $n=1$.
\item If $l_{X_i}(G_i)=6$ then $X_i$ is $6$-primitive so $G_i \in \{\Alt(5),\Alt(6)\}$. If $G_i=\Alt(5)$ then $G \in \{\Alt(5),\Sym(5)\}$ and $n=1$, otherwise if $G_i=\Alt(6)$ then $X_i$ is a subgroup of $\Aut(\Alt(6))$ containing $\Alt(6)$, so either $X_i/G_i$ has order $\leq 2$, and if this is the case $G=X_i$ and $n=1$, or $X_i=\Aut(\Alt(6))$ and $X_i$ is a $3$-sum group since it admits $C_2 \times C_2$ as an epimorphic image. Since $X_i$ is a quotient of $G$, this implies that $\sigma(G)=3$. This contradicts $\sigma(G) \geq l_{X_i}(G_i)$.
\item If $l_{X_i}(G_i)=10$ then $$X_i \in \{\Alt(5),\Sym(5),PSL(2,9),PGL(2,9),$$ $$\Sym(6),M_{10},P \Gamma L(2,9),\Alt(10),\Sym(10)\}.$$Now since $l_{X_i}(G_i) \leq \sigma(G)$, and since $X_i$ is a quotient of $G$ we cannot have $X_i=P \Gamma L(2,9)$ because $\sigma(P \Gamma L(2,9))=3$; it follows that $X_i/G_i \in \{1,C_2\}$ so $G=X_i$ and $n=1$.
\end{itemize}
We deduce that if $G_i$ is non-abelian then either $n=1=i$ or $l_{X_i}(G_i) \geq 17$. In particular if $G_1$ is non-abelian and $\sigma(G) \leq 33$ then $n=1$ since $33 \geq \sigma(G) \geq \sum_{i=1}^n l_{X_i}(G_i) \geq 17n$. Let us suppose that $G_1$ is abelian. If $n \geq 2$ then $l_{X_i}(G_i) \geq 17$ for $i \geq 2$, so $|G_1| \leq \sigma(G)-17$ and $\sigma(G)-17 \geq |G_1| > \frac{1}{2} \sigma(G)$, thus $\sigma(G) > 34$. \\
We deduce that if $G$ is $\sigma$-elementary, non-abelian and non monolithic then $\sigma(G) \geq 34$.
\end{proof}
Now the study of $n$-primitive groups with $3 \leq n \leq 25$ is easy. Since abelian non-cyclic $\sigma$-elementary groups are $(p+1)$-elementary and isomorphic to $C_p \times C_p$ for some prime $p$, we will look for non-abelian $\sigma$-elementary groups. Since every non-abelian $n$-elementary group with sum $n \leq 25$ is primitive of degree $\leq n-1$, it sufficies to look at the tabulars of section \ref{prim} to obtain the following tabular, in which there are for every sum $n$ all the $n$-elementary groups.
$$
\begin{tabular}{|c|c|}
\hline $\mbox{sum}$ & $\mbox{groups}$ \\ \hline $3$ & $C_2 \times C_2$ \\ \hline $4$ & $C_3 \times C_3,\Sym(3)$ \\ \hline $5$ & $\Alt(4)$ \\ \hline $6$ & $C_5 \times C_5,D_{10},AGL(1,5)$ \\ \hline $7$ & $\emptyset$ \\ \hline $8$ & $C_7 \times C_7,D_{14},7:3,AGL(1,7)$ \\ \hline $9$ & $AGL(1,8)$ \\ \hline $10$ & $3^2:4,AGL(1,9),\Alt(5)$ \\ \hline $11$ & $\emptyset$ \\ \hline $12$ & $C_{11} \times C_{11},11:5,D_{22},AGL(1,11)$ \\ \hline $13$ & $\Sym(6)$ \\ \hline $14$ & $C_{13} \times C_{13},D_{26},13:3,13:4,13:6,AGL(1,13)$ \\ \hline $15$ & $SL(3,2)$ \\ \hline $16$ & $\Sym(5),\Alt(6)$ \\ \hline $17$ & $2^4:5,AGL(1,16)$ \\ \hline $18$ & $C_{17} \times C_{17},D_{34},17:4,17:8,AGL(1,17)$ \\ \hline $19$ & $\emptyset$ \\ \hline $20$ & $C_{19} \times C_{19},AGL(1,19),D_{38},19:3,19:6,19:9$ \\ \hline $21$ & $\emptyset$ \\ \hline $22$ & $\emptyset$ \\ \hline $23$ & $M_{11}$ \\ \hline $24$ & $C_{23} \times C_{23},D_{46},23:11,AGL(1,23)$ \\ \hline $25$ & $\emptyset$ \\ \hline
\end{tabular}
$$

\textit{Martino Garonzi, Dipartimento di Matematica Pura ed Applicata, Via Trieste 63, 35121 Padova, Italy. E-mail address: mgaronzi@studenti.math.unipd.it}
\end{document}